\documentclass[12pt,a4paper,final]{article}
\usepackage[margin=25.4mm]{geometry}

\usepackage{epsfig,amsmath,amsfonts,bef_alex,amssymb,todonotes,bm,relsize}
\usepackage[utf8]{inputenc}

\usepackage{color}

\numberwithin{equation}{section}

\usepackage[notref,notcite,color]{showkeys}

\newcommand{\III}{{]{-}\infty,0[}}
\renewcommand{\ti}{{\times}}
\newcommand{\rhoo}{\rho_{\mathrm{memb}}}
\newcommand{\rhooo}{\rho_{\mathrm{bulk}}}
\newcommand{\kkk}{\kappa}
\newcommand{\llla}{\big\langle\!\!\big\langle}
\newcommand{\rrra}{\big\rangle\!\!\big\rangle}
\newcommand{\norm}{\bm|\!\!\,\bm|\!\!\,\bm|}

\def\Dot#1{\overset{\text{\LARGE.}}{#1}}
\def\DDot#1{\overset{\text{\LARGE.\hspace{-0.1em}.}}{#1}}

\begin{document}

\title{A rigorous derivation and energetics
       of\\ a wave equation with fractional damping%
 \thanks{The research was
   partially supported by Deutsche Forschungsgemeinschaft through
   SFB\,1114 \emph{Scaling Cascades in Complex Systems}, Project Number 235221301, 
    via the
   subprojects C02 \emph{Interface Dynamics: Bridging Stochastic and
     Hydrodynamics Descriptions}
   and C05 \emph{Effective Models for Materials and Interfaces
  with Multiple Scales}.}}

\author{Alexander Mielke\thanks{Weierstraß-Institut f\"ur Angewandte
    Analysis und Stochastik, Mohrenstr.\,39, 10117 Berlin und Humboldt
  Universit\"at zu Berlin, Institut f\"ur Mathematik, Berlin,
  Germany.}, Roland R. Netz\thanks{Freie Universität Berlin, Fachbereich
 Physik, Berlin, Germany.}, and Sina Zendehroud$^\ddag$}

\date{24.\ April 2020}

\maketitle

\centerline{\itshape Dedicated to Matthias Hieber on the occasion to his
  sixtieth birthday}

\begin{abstract}
We consider a linear system that consists of a linear wave equation on a
horizontal hypersurface and a parabolic equation in the half space
below. The model describes longitudinal elastic  waves in organic monolayers 
at the water-air interface, which is an experimental setup that is relevant
for understanding wave propagation in biological membranes.
We study the scaling regime where the relevant horizontal
length scale is much larger than the vertical length scale and provide
a rigorous limit leading to a fractionally-damped wave equation for the
membrane. We provide the associated existence results via linear
semigroup theory and show convergence of the solutions in the scaling
limit. Moreover, based on the energy-dissipation structure for the
full model, we derive a natural energy and a natural dissipation
function for the fractionally-damped wave equation with a time
derivative of order 3/2.\smallskip

\smallskip

\noindent
\emph{MOS:}
35Q35  
35Q74  
74J15  
\end{abstract}

\section{Introduction}
\label{se:intro}

This work is stimulated by the physical models investigated in
\cite{KSSN17NFWE, KSSN17NFWEs}, where longitudinal elastic waves of a
membrane are coupled to viscous fluid flow in the enclosing half
space. The aims are to understand the damping of the elastic
waves through the coupling to the viscous fluid, on the one hand,
and to explain the appearance of the non-classical dispersion relation, on the
other hand. Denoting by $k \in \R^{d-1}$ the horizontal wave vector and
$\omega\in \R$ the angular frequency, the classical elastic wave satisfies a dispersion
relation $\omega^2 \approx |k|^2$, while the longitudinal
pressure waves, here referred to as Lucassen waves (cf.\ \cite{Luca68LCW1}), 
satisfy $|\omega|^{3/2} \approx |k|^2$  such that the wave speed
$c(k)=\omega(k)/|k|$ depends on $k$.

The class of Lucassen waves attracted considerable attention over the
last decade due to its biophysical relevance in living organisms,
where the transmission of information over biologically relevant
distances and time scales is fundamental. The standard model
describes the propagation of signals on the vast network of nerve
cells via a purely electrical mechanism, unable to explain a number of
nonelectric phenomena, like the effectiveness of anesthetics scaling
with their solubility in lipid membranes \cite{Meye99TA, Over01SNBA}
or the lower heat dissipation of a nerve in contrast to an electrical
cable \cite{Tasa92HPAP}. As it is known that, alongside the electrical
signal, a mechanical displacement travels along the nerve fiber
\cite{KKOS07MSAA, ElhMac15MSWA}, there is a need for a more complete
model incorporating these aspects. On the one hand, experimental
scientists (see \cite{GBWS12P2DP, ShrSch14ETDS}), using a lipid
monolayer spread at the air-water interface as a minimal model, have
shown that indeed pressure waves can propagate in such systems. On the
other hand, from a more theoretical viewpoint, all possible
surface-wave solutions for a visco-elastic membrane atop a half space
of viscous fluid have been determined in \cite{KapNet15MSWS},
including the experimentally observed Lucassen waves and their
dispersion relations of the type $|\omega|^{3/2}\approx |k|^2$. In
particular, a fractionally-damped wave equation was derived for
describing the Lucassen waves efficiently in \cite{KSSN17NFWE,
  KSSN17NFWEs}.  The biophysical relevance of Lucassen wave is
demonstrated by the fact that the wave solutions depend directly on
the lateral membrane compressibility $\kkk$.  For example, adsorption
of lipophilic substances, like anesthetics, into the membrane
presumably alter $\kkk$ and, as a consequence, are expected to change
the wave propagation properties.  In addition to that, at large
amplitudes, the pressure pulse locally modifies the compressibility
$\kkk$ and thereby significantly increases the propagation distance
\cite{ShrSch14ETDS}.  This non-linear property suggests an all-or-none
behavior, which indeed is observed in nerve pulse propagation.

Here we want to understand this phenomenon using the mathematically most
simple model, which is given by the following coupled system:
\begin{subequations}
\label{eq:System}
\begin{align}
&\label{eq:System.a}
  \rhoo \DDot{U}= \kkk \Delta_x U - \mu \pl_z v &&\text{for } t>0, \ x\in \Sigma,\\
&\label{eq:System.b}
  \Dot U(t,x)= v(t,x,0)&& \text{for } t>0, \ x\in \Sigma,\\
& \label{eq:System.c}
 \rhooo \Dot v = \mu \Delta_{x,z} v &&\text{for } t>0,\ (x,z)\in
 \Omega:=\Sigma {\ti} {]{-}\infty,0[}.
\end{align}
\end{subequations}
In the physical setup of \cite{KSSN17NFWE, KSSN17NFWEs} the domain
$\Sigma = \R^1$ denotes the membrane and $U(t,x)\in \R$ denotes
the horizontal  displacement (longitudinal motion) of the
membrane. The half space $\Omega = \Sigma \ti \III \subset \R^d$ is filled by a
viscous fluid whose horizontal velocity component is $v(t,x,z) \in
\R$ (pure shear flow). Condition \eqref{eq:System.b} is a no-slip
condition for the fluid along the membrane, while the induced stress
of the sheared fluid is included in \eqref{eq:System.a} via ${-} \mu
\pl_z v(t,x,0)$.

For mathematical purposes we can allow $\Sigma \subset \R^{d-1}$, but to
avoid any complications with boundary conditions we assume that
$\Sigma$ is of the form
\begin{equation}
  \label{eq:Sigma.Form}
    \Sigma = \R^k \ti \big(\R_{\!/(\ell\Z)}\big)^n \quad \text{ with } \ell>0
  \text{ and } k+n=d{-}1.
\end{equation}
In particular, $\Sigma$ is an additive group and \eqref{eq:System} is
translation invariant. In Section \ref{se:Model} we first show that
the system has the natural energy
\begin{equation}
  \label{eq:I.Energy}
    \calE(U,\Dot U,v)= \int_\Sigma\Big\{\frac\rhoo2 \Dot U{}^2 + \frac\kkk2
  |\nabla U|^2\Big\} \dd x + \int_\Omega \frac\rhooo2 v^2 \dd z \dd x
\end{equation}
as a Lyapunov function. This shows that the function space
$\bfH:= \rmH^1(\Sigma)\ti \rmL^2(\Sigma) \ti \rmL^2(\Omega)$  is the
natural state space. Note that this includes periodic boundary
conditions for $x \in \Sigma = \R^k \ti \big(\R_{\!/(\ell\Z)}\big)^n$.

Moreover, we discuss suitable scalings of time
$t$, the horizontal variable $x\in \R^{d-1}$, and the vertical variable $z\in
\III $. We can renormalize all constants such that the system of equations  takes
the form
\begin{subequations}
\label{eqI:SystemN}
\begin{align}
&\label{eqI:SystemN.a}
  \DDot U= \Delta_x U -  \pl_zv|_{z=0} &&\text{for } t>0, \ x\in \Sigma,\\
&\label{eqI:SystemN.b}
  \Dot U= v|_{z=0}&& \text{for } t>0, \ x\in \Sigma,\\
& \label{eqI:SystemN.c}
  \Dot v = \eps^2 \Delta_x v + \pl_z^2 v &&\text{for } t>0,\ (x,z)\in
 \Omega,
\end{align}
\end{subequations}
with the parameter $\eps = \mu/\sqrt{\rhoo k\,}$\,. The essential point is
here that the scaling of the horizontal variable $x\in \R^{d-1}$ is different from
the vertical variable $z \in \III$, thus breaking the isotropy of the
diffusion $\mu\Delta_{x,z} v$ in \eqref{eq:System.c}.

Our interest lies in the case $\eps\to 0$. Indeed,
in Section \ref{su:LimitMOdel} we simply set $\eps=0$ in
\eqref{eqI:SystemN.c} and show that this limit allows us to solve the
scalar one-dimensional (!!) heat equation $\Dot v = \pl_z^2 v$ on
$\III$ for each $x\in \Sigma$ independently. Assuming the initial
condition $v(0,x,z)=0$ and using the Dirichlet boundary condition
$v(t,x,0)=\Phi(t,x)$, the stress $\pl_z v(t,x,0)$ can be explicitly
expressed via the heat kernel, namely
\begin{equation}
  \label{eq:I.parDtN}
    \pl_z v(t,x,0)= \int_0^t \frac1{\sqrt{\pi(t{-}\tau)}}
  \: \Dot \Phi(\tau,x) \dd \tau,
\end{equation}
see \eqref{eq:DiffHL4}. It is this one-dimensional parabolic
Dirichlet-to-Neumann map that introduces the fractional damping into
the wave equation. In particular, denoting the fractional (Caputo)
derivative of order $\alpha\in {]0,1[}$ of the function $g$ with
$g(0)=0$ by
\begin{equation}
  \label{eq:Caputo}
    {}^\rmC\rmD^\alpha g: \ t \ \mapsto \ \  \frac1{\Gamma(1{-}\alpha)} \int_0^1
  \frac1{(t{-}\tau)^\alpha} \: \Dot g(\tau)\dd \tau,
\end{equation}
we see that the mapping in \eqref{eq:I.parDtN} takes the form $\pl_z
v(t,x,0)= (\,{}^\rmC\rmD^{1/2} \Phi)(t,x)$, which is a Caputo derivative of order
$1/2$.

Indeed, if we solve \eqref{eqI:SystemN} with $\eps=0$ and
$v(0,x,z)=0$, then we can eliminate $v$ totally by exploiting
\eqref{eq:I.parDtN} with $\Phi=\Dot U$,  and $U$ has to solve
\begin{equation}
  \label{eq:I.FWE}
    \DDot U(t,x) + \int_0^t \frac1{\sqrt{\pi(t{-}\tau)}}
  \: \DDot U(\tau,x) \dd \tau  = \Delta_x U(t,x) \quad \text{for } t>0,
  \ x\in \Sigma.
\end{equation}
This is a fractionally-damped wave equation where the damping is
generated by a fractional Caputo derivative of order $3/2$, and this
fractional derivative acts locally with respect to the space variable $x \in
\Sigma$.

In Section \ref{su:Dispersion} we follow the approach in
\cite{KapNet15MSWS, KSSN17NFWE, KSSN17NFWE} and discuss the dispersion
relations for our normalized system \eqref{eqI:SystemN} and show that,
for the limit case $\eps=0$, the dispersion relation reads $(\ii \omega)^2 +
(\ii \omega)^{3/2} + |k|^2=0$, where $\Im \omega\geq 0$ is enforced by
the stability through the Lyapunov function $\calE$. Hence, for small
$|k|$ we obtain the new dispersion relation
\[
  \omega(k)=\Big(\pm \frac{\sqrt3}2 + \frac\ii 2 \Big) \,|k|^{4/3}
  \ + \ O(|k|^2)_{k\to 0}.
\]

There is a rich mathematical literature on linear and nonlinear
partial differential equations involving fractional time derivatives,
see e.g.\ \cite{VerZac08LFCS, VerZac10ABDS, PrVeZa10WLBN, KSVZ16DETF,
  VerZac17SIBT, Akag19FFDS}. Our focus is different, because we want to
show that \eqref{eq:I.FWE} appears as a rigorous limit for
$\eps\to 0^+$ in \eqref{eqI:SystemN}. For this, in Section
\ref{se:ConvergenceSG} we develop the linear semigroup theory by
showing that the semigroups $\ee^{t A_\eps}:\bfH \to \bfH$ exist for all
$\eps\geq 0$ and are bounded in norm by $C(1{+}t)$.
In Theorem \ref{th:StrongCvgSol} we establish the strong convergence
$\ee^{t A_\eps} w_0 \to \ee^{t A_0} w_0$ for $\eps\to 0^+$, which holds
for all $t>0$ and $w_0 \in \bfH$.  For more regular initial conditions,
$w_0$ we obtain the quantitative estimate
\[
  \| \ee^{t A_\eps} w_0 - \ee^{t A_0} w_0\|_\bfH \leq \sqrt{\eps \,t}\:
  (2.3{+}t)^2 \: \big( \|w_0\|_{\bfH} + \|\nabla_x w_0\|_{\bfH}\big).
\]

In Section \ref{se:EnergyDiss} we return to the energetics and the
dissipation for the damped wave equation. By starting from the natural
energy and dissipation in the PDE system \eqref{eqI:SystemN} with
$\eps=0$ and the explicit solution for $v(t,x,z)$ in terms of $\Dot
U(\tau,x)$, we obtain a natural energy functional $\bfE$ for the
fractionally-damped wave equation that is non-local in time:
\begin{subequations}
  \label{eq:I.EnergDiss}
\begin{align}
  \bfE(U(t), \big[\Dot U(\cdot)\big]_{[0,t]}) = \int_\Sigma
  \Big\{ &\frac12\Dot U(t,x)^2 + \frac12|\nabla U(t,x)|^2 \nonumber
  \\
  & +\int_0^t\!\int_0^t \!
  \frac1{4\sqrt\pi(2t{-}r{-}s)^{3/2}} \:\Dot U(r,x) \Dot
  U(s,x) \dd s \dd r \Big\} \dd x ,
\end{align}
where $\big[\Dot U(\cdot)\big]_{[0,t]}$ indicates the dependence on
$\Dot U(s)$ for $s\in [0,t]$.  For solutions $U$ of \eqref{eq:I.FWE},
we obtain an energy--dissipation balance with a non-local dissipation:
\begin{align}
  \frac{\rmd}{\rmd t} \bfE(U(t), \big[\Dot U(\cdot)\big]_{[0,t]})  = -
  \int_\Sigma \int_0^t\int_0^t \frac1{\sqrt\pi(2t{-}r{-}s)^{1/2}}
  \:\DDot U(r,x) \DDot   U(s,x) \dd s \dd r \: \dd x.
\end{align}
\end{subequations}
Related results are obtained in
\cite{VerZac08LFCS, GYCCEO15MMSD, VerZac15ODET}, but there typically
only energy-dissipation inequalities are derived. It is
surprising to see that the two non-local kernels   in
\eqref{eq:I.EnergDiss} that depend on
$t{-}r,\ t{-}s\in [0,T]$ are only depending on the sum $(t{-}r) +
(t{-}s) = 2t{-}r{-}s$, which derives from very specific scaling
properties of the heat kernel.

\section{The formal modeling}
\label{se:Model}

In this section we describe the formal modeling, including the energy
functional, the scalings and the derivation of the fractionally-damped
wave equation as the scaling limit.

\subsection{The energy functional and the state space}
\label{su:LinModel}

We return to the full system \eqref{eq:System} and observe that it has
the form of a damped Hamiltonian system with the total energy
$\calE(U,\Dot U, v)$ given in \eqref{eq:I.Energy}.  Indeed, taking the
time derivative along solutions $t\mapsto (U(t),v(t))$ of
\eqref{eq:System} we find
\begin{align*}
&\frac{\rmd}{\rmd t} \calE(U(t),\Dot U(t), v(t)) = \int_\Sigma \big(\rhoo
 \Dot U  \DDot U +\kkk \nabla_{\!x} \Dot U \cdot\nabla_{\!x} U \big)
                 \dd x + \int_\Omega
\rhooo v\,\Dot v  \dd x \dd z \\
&\overset{\text{\eqref{eq:System.a},\eqref{eq:System.c}}}
=\int_\Sigma U_t \mu v_z \dd x +
 \int_\Omega  \mu \,v\,\Delta_{x,z} v \dd x  \dd z
 \overset{\text{\eqref{eq:System.b}}}=
- \int_\Omega  \mu ( |\nabla_{\!x} v|^2 {+}v_z^2) \dd x \dd z \leq 0.
\end{align*}
Here we used that the integration by parts $\int_\Sigma   \nabla_{\!x}
\Dot U \cdot\nabla_{\!x} U \dd x = - \int_\Sigma \Dot U \Delta_x U\dd x $
does not generate boundary terms because $\Sigma $ has the form
\eqref{eq:Sigma.Form}.

Thus, $\calE$ acts as a Lyapunov function and it is a bounded
quadratic form on the Hilbert space $\bfH= \rmH^1(\Sigma)\ti \rmL^2(\Sigma)\ti
\rmL^2(\Omega)$, which we consider as the basic state space for our
problem.  In Section
\ref{se:ConvergenceSG} we will show that \eqref{eq:System} has a
unique solution for each initial value $w_0=(U(0),\dot U(0),v(0)) \in
\bfH$.

More precisely, the system \eqref{eq:System} can be written as a damped
Hamiltonian system for the states $X=(U,P,p)$ where $P=\rhoo \Dot
U$ and $p=\rhooo v$. With $\calH(U,P,p)=\calE(U,\frac1\rhoo
P,\frac1{\rhooo} p)$ we have
\[
  \left(\!\!\ba{c}\Dot U\\ \Dot P\\ \Dot p \ea\!\!\right) = 
  \big( \mathbb J {-} \mathbb
  K\big) \rmD \calH(U,P,p) \ 
  \text{with }\mathbb J=\bma{ccc}0&I&0\\\!\!\!{-}I&0&0\\ 0&0&0 \ema
  \text{ and }\mathbb K =\bma{ccc}0&0&0\\0&*&\mu\pl_z\Box|_{z=0}\!\!\\
  0&*&{-}\mu\Delta_{x,z}\! \ema\!\!,
\]
where $\mathbb K: \mafo{dom}(\mathbb K) \subset \bfX\to
\bfX:=\rmL^2(\Sigma) \ti \rmL^2(\Sigma)\ti \rmL^2(\Omega)$ is defined
as the
self-adjoint, unbounded operator induced by the quadratic dissipation
potential
\[
  \calR^*(\Pi,\Xi,\xi):=\frac\mu2 \int_\Omega
  |\nabla_{x,z}\xi|^2 \dd x \dd z + \chi_*(\Xi,\xi), \quad
  \text{ with }\chi_*(\Xi,\xi)= \left\{\ba{cl} 0 & \text{if
    }\Xi=\xi|_{z=0}, \\ \infty &\text{else}. \ea \right.
\]

\subsection{A long-wave scaling}
\label{su:LongWaveSc}

To obtain a first understanding of the different scaling of horizontal
and vertical spatial variables,  we study the long-wave scaling for the
wave equation. This mean that we scale the horizontal space variable
$x$ and the time variable $t$ with the same factor $\delta>0$. For the moment we
assume that the membrane constants $\rhoo$ and $k$ are given and of
order 1, while $\rhooo$ and $\mu$ are much smaller. More general
scalings are discussed in the following subsection.

Without loss of generality we keep $U$ fixed and  obtain velocities
$\Dot U$ of order $\delta$. Hence, to keep the no-slip condition, we also
need to rescale $v$ by a factor $\delta$. The main point is that we want
$z$ to be rescaled by a smaller factor, let us say $\delta^\alpha$ with
$\alpha \in {[0,1[}$. This implies that $\pl_z v$ scales like
$\delta^{1+\alpha}$. Thus, to treat the coupling term
$\mu\pl_z v|_{z=0}$ of the same order as $\DDot U$ and $\nabla_x U$, we
need to assume that $\mu$ also scales with $\delta$, namely like
$\delta^{1-\alpha}$. Finally, we also assume the appropriate scaling for
$\rhooo$, namely
\[
\hat x = \delta \,x, \quad \hat t = \delta\, t, \quad \hat z = \delta^\alpha\, z,
\quad \hat U = U,\quad  v = \delta\, \hat v , \quad
 \mu = \delta^{1-\alpha} \hat\mu, \quad
\rhooo =\delta^\alpha \hat{\ }\hspace{-1ex}\rhooo .
\]
Hence, this long-wave scaling with small $\delta$ is indeed suitable, if the bulk
quantities $\rhooo$ and $\mu$ are much smaller than the membrane
quantities $\rhoo$ and stiffness $k$.

Inserting these scalings (and dropping the hats)
we find the transformed system
\begin{align}
\label{eq:ScSystem}
  \begin{aligned}
                  &\rhoo \DDot U= \kkk \Delta_x U - \mu \pl_z v|_{z=0}
  \text{ and }
  \Dot U= v|_{z=0} &\text{ on }&\Sigma,
  \\
  &  \rhooo \Dot v = \mu (\delta^{2{-}2\alpha} \Delta_x {+} \pl_z^2)v
     &\text{ in } &\Omega.
 \end{aligned}
\end{align}
Here, the case of small $\delta$ is relevant, and in the limit
$\delta\to 0^+$ we obtain the fractionally-damped wave equation.

\subsection{Non-dimensionalizing by a general scaling}
\label{su:GeneralScaling}

We fully non-dimensionalize the system by considering general
rescalings, where we scale $x$, $z$, and $t$ independently:
\[
  \hat x = a \,x, \quad \hat t = b \,t, \quad \text{and } \hat z = c\, z,
\]
but do not assume any scaling on the material parameters $\rhoo,\ \kkk,\
\rhooo$, and $\mu$. We keep $\hat U=U$ (which is always possible by
linearity), but need to rescale $v= b\,\hat v$ to transform  the no-slip
condition $\Dot U = v|_{z=0}$ into $\pl_{\hat t} \hat U = \hat
v|_{\hat z=0}$. The transformed equations read (after dropping the
hats) as follows
\[
  \rhoo b^2 \DDot U = \kkk a^2 \Delta_{x}U - \mu bc\,\pl_z v|_{z=0}
  \text{ on }\Sigma, \quad
  \rhooo b^2 \Dot v =  \mu bc^2 \pl_z^2 v + \mu a^2b \Delta_x v  \text{ in
  }\Omega.
\]
Dividing the equations by $\rhoo b^2$ and $\rhooo b^2$ respectively, we
can equate the first three of the four coefficients to 1, namely
\[
\frac{\kkk a^2}{\rhoo b^2} = 1,\quad \frac{\mu c}{\rhoo b} = 1, \quad
\frac{\mu c^2 }{\rhooo b}=1.
\]
We obtain the solution
\[
a^2 =\frac{\rhooo^2\mu^2}{\rhoo^3 \kkk}, \quad b = \frac{\rhooo
  \mu}{\rhoo^2}, \quad \text{and } c = \frac\rhooo\rhoo,
\]
and the remaining fourth coefficient reads
\[
  \eps:= \Big(\frac{\mu a^2}{\rhooo b}\Big)^{1/2}
  =
  \frac{a}{c}= \frac{\mu}{\sqrt{\rhoo \kkk}}.
\]
The non-dimensionalized coupled system now reads
\[
  \DDot U = \Delta_{x}U - \pl_z v|_{z=0} \text{ and } \Dot U=v|_{z=0}
  \text{ on }\Sigma, \quad   \Dot v =   \eps^2 \Delta_x v
  +\pl_z^2 v \text{ in }\Omega,
\]
which is exactly the renormalized system \eqref{eqI:SystemN}, which is
studied subsequently.

Hence, the system \eqref{eq:System} has a unique
non-dimensional parameter $\eps= \mu/\sqrt{\rhoo\, k\,}$ that
describes the effective anisotropy of the diffusion in the bulk
$\Omega$. Subsequently, we are interested in the case of very small
$\eps$ and indeed in the limit $\eps\to 0^+$.

We can interpret $\eps$ as the ratio of three different length
scales. Choosing an arbitrary time scale $t_0>0$, we have the
diffusion length $\ell_\mafo{diff}$, the ``equivalent membrane thickness''
$\ell_\mafo{thick}$, and the membrane travel length $\ell_\mafo{trav}$
given by
\[
  \ell_\mafo{diff}(t_0) = \Big( \frac{\mu t_*}{\rhooo}\Big)^{1/2}, \quad
  \ell_\mafo{thick}= \frac{\rhoo}\rhooo, \quad \text{and }
  \ell_\mafo{trav}(t_0)= t_0 c_\mafo{memb} = t_0 \Big(\frac
  \kkk\rhoo\Big)^{1/2},
\]
where $c_\mafo{memb}$ is the wave speed in the undamped membrane.
Now our dimensionless parameter $\eps$ is
given by
\[
  \eps = \frac{\ell_\mafo{diff}(t_0)^2} {\ell_\mafo{thick}\:
    \ell_\mafo{trav}(t_0)} \quad \text{for all }t_0>0\:.
\]

To make the definition even more intrinsic, we may choose $t_0$ as a
characteristic time $t_*$ for the system. We ask that the time
$t_*$ is chosen such that the corresponding diffusion length scale
equals the equivalent membrane thickness, viz.\
$\ell_\mafo{diff}(t_*)=\ell_\mafo{thick}$. This yields
\[
  t_*=  \frac{\rhoo^2}{ \mu \rhooo } \quad \text{and}
   \quad
   \ell_\mafo{trav}(t_*) = t_* c_\mafo{memb} 
   = \frac{ \rhoo^{3/2} \kkk^{1/2}}{\mu \rhooo } .
\]
The scalings of time and horizontal and vertical lengths are now given
as
\[
  t = t_* \:\hat t, \quad x =  \ell_\mafo{trav}(t_*)\:\hat
  x,\quad z = \ell_\mafo{thick} \:\hat z.
\]
This leads to the final relation
\[
\eps = \frac{\ell_\mafo{thick}}{ \ell_\mafo{trav}( t_*) } =\frac{ \mu}{
\sqrt{\rhoo \kkk}} \:.
\]

Typical parameters for the experimental setup consisting of a lipid
monolayer, such as DPPC at the water-air interface, are
$\rhoo = 10^{-6}~\mathrm{kg / m^2}$,
$\mu = 10^{-3}~\mathrm{Pa\,s} = 10^{-3}~\mathrm{kg/(s\,m)}$, and
$\kkk = 10^{-2}~\mathrm{N / m}$, where for $\rhoo$ the surface excess
mass density was used.  These parameters yield $\eps = 10$.  Although
this value is not small, it does not contradict our argumentation. As
shown in \cite{KapNet15MSWS}, different waves can coexist in such a
system, the longitudinal capillary waves with
$|\omega|^{3/2} \approx |k|^2$ being only one of them. In particular
it is interesting that the dispersion of this wave, which has been
known in the literature since \cite{Luca68LCW1}, follows from our
general calculation as a rigorous limit.

\subsection{The limit model and the fractionally-damped
  wave equation}
\label{su:LimitMOdel}

We now study the limit equation by setting $\eps=0$ in the
rescaled system
\eqref{eqI:SystemN}. The justification of taking this limit is given
in the following section.

After setting $\eps=0$ we obtain the system
\begin{subequations}
\label{eq:LiSystem}
\begin{align}
&\label{eq:LiSystem.a}
  \DDot U (t,x)= \Delta_x U(t,x) - \pl_z v(t,x,0) &&\text{for } t>0, \
                                                     x\in \Sigma,\\
&\label{eq:LiSystem.b}
  \Dot U(t,x) = v(t,x,0) && \text{for } t>0, \ x\in \Sigma,\\
& \label{eq:LiSystem.c}
 \Dot v(t,x,z) = \pl_z^2 v(t,x,z) &&\text{for } t>0,\ (x,z) \in \Omega.
\end{align}
\end{subequations}
The point is now that the equation for \eqref{eq:LiSystem.c} can be
solved explicitly by the use of the properly rescaled one-dimensional
heat kernel
$H(t,y)=(4\pi t)^{-1/2} \ee^{-y^2/(4t)}$. Note that $x\in \Omega$
appears now as a parameter only, since the diffusion in $x$-direction
is lost.

The solution of \eqref{eq:LiSystem.c} with the boundary condition
\eqref{eq:LiSystem.b} and the initial condition
$v(0,x,z)=0$ takes the explicit form
\[
v(t,x,z)= \int_0^t 2 \pl_z H(t{-}\tau,z) \, \Dot U(\tau,x) \dd \tau,
\]
see Section \ref{su:HalfLine} for a derivation. Taking the derivative
with respect to $z$ and using that the heat kernel $H$ satisfies
$\pl_z^2 H= \pl_t H$ we obtain
\[
  \pl_z v(t,x,z)= \int_0^t 2\pl_t H(t{-}\tau,z)\,\Dot U(\tau,x) \dd \tau =
  \int_0^t 2 H(t{-}\tau,z)\,\DDot U(\tau,x) \dd \tau,
\]
where for the integration by parts in the last identity we exploited
$\Dot U(t,x,0)=v(t,x,0)=0$ and $H(0,z)=0$ for $z<0$. Thus, evaluation at
$z=0$ and using $H(t,0)=(4\pi t)^{-1/2}$, the coupling
term in \eqref{eq:LiSystem.a} reduces to
\begin{equation}
  \label{eq:Memory}
  \pl_z v(t,x,0)= \int_0^t 2H(t{-}\tau,0) \DDot U(\tau,x)\dd \tau =
\int_0^\tau \frac1{\sqrt{\pi(t{-}\tau)}} \: \DDot U(\tau,x) \dd \tau.
\end{equation}

Through these formulas we see how kinetic energy is moved from the
membrane via the no-slip condition \eqref{eq:LiSystem.b} into the
one-dimensional diffusion equation. Through the memory kernel in
\eqref{eq:Memory} the energy is restored partially in a delayed
fashion, which leads to a fractional damping, here of order
$3/2$. Because in the bulk $\Omega= \Sigma \ti \III$ there is no
coupling between different points $x \in \Sigma$, this damping is
non-local in time but local with respect to $x\in \Sigma$.

Joining everything we see that the limiting system
\eqref{eq:LiSystem} contains the fractionally-damped wave equation
\begin{equation}
  \label{eq:FDWE}
  \DDot U(t,x) + \int_0^t \frac1{\sqrt{\pi(t{-}\tau)}}\: \DDot
U(\tau,x)\dd \tau = \Delta U(t,x) \quad \text{on } \Sigma.
\end{equation}
An analysis concerning existence of solutions and concerning the
energetics is given in Section \ref{se:EnergyDiss}.

\subsection{The dispersion relations}
\label{su:Dispersion}

Following \cite{KSSN17NFWE} we consider special solutions of
\eqref{eq:SystemN} obtained by a Fourier ansatz. For the temporal
growth factor $\mu=\ii \omega\in \C$ with $\Re \mu \leq 0$ and the wave vector
$k \in \R^{d-1}$  we set
\[
 (U(t,x),V(t,x), v(t,x,z)) = \ee^{\mu t + \ii k \cdot x} \big( a,b,
 c\,\ee^{\gamma z} \big)
\]
with $a,b,c,\gamma \in \C$ and $\Re \gamma>0$. From $V=U_t$ we obtain
$b=a\mu$, while $v|_{z=0}=V$ implies $c=b=a\mu$. Finally, we have to satisfy
the membrane equation $U_{tt} = \Delta_{x} U - v_z|_{z=0}$ and the diffusion
equation $\pl_t v= \eps^2 \Delta_x v+\pl_z^2 v$, which leads to the
algebraic relations (since only $a\neq 0$ is interesting)
$  \mu^2 = -| k |^2 - \mu \gamma$ and $\mu = -\eps^2 | k |^2 + \gamma^2$.
As in \cite{KSSN17NFWE} we eliminate  the variable $\gamma$ and
obtain the dispersion relation
\[
  0=\Gamma(\mu,  k )= \big(\mu^2{+}|k|^2\big)^2 - \eps^2 \mu^2|k|^2 -
  \mu^3,
\]
where we still need to be careful to satisfy $\Re \mu \leq 0$ and $\Re
\gamma >0$ with $\gamma^2=\mu {+} \eps^2|k|^2$.

For short waves, i.e.\ $|k|\gg1$, we obtain the expansion
\[
  \mu= - \ii |k| - \frac{|k|}2 \:\Big( \eps^2 -
  \frac{\ii}{|k|}\Big)^{1/2} + \text{h.o.t.}
\]
In the case $\eps>0$ this means that short waves travel at speed 1,
but are damped proportional to $|k|$. The limit $\eps=0$
leads to a significantly smaller damping, namely one of order $|k|^{1/2}$.

As expected due to the scaling discussed in the previous subsections,
the case of long waves, i.e.\ $|k| \ll 1$, is not so sensitive with
respect to $\eps$. For all $\eps\geq 0$ we find the expansion
\[
\mu= - \Big( \frac12 \pm \frac{\ii \sqrt3}2\Big) |k|^{4/3} + \text{h.o.t.}
\]
In particular, we find that the waves slow down for $|k|\to 0$,  because
the wave speed takes the form
$c(k) =\Im(\mu(k)/|k|) = \pm |k|^{1/3}\sqrt3/2 + $h.o.t.\ Moreover,
the damping is very low, because it is proportional to $|k|^{4/3} $.

\section{Convergence result for the semigroup}
\label{se:ConvergenceSG}

From now on it suffices to consider the rescaled system, where $\eps>0$
appears as the only small parameter:
\begin{subequations}
\label{eq:SystemN}
\begin{align}
&\label{eq:SystemN.a}
   \DDot U= \Delta_x U -  \pl_zv|_{z=0} &&\text{for } t>0, \ x\in \Sigma,\\
&\label{eq:SystemN.b}
  \Dot U= v|_{z=0}&& \text{for } t>0, \ x\in \Sigma,\\
& \label{eq:SystemN.c}
  \Dot v = \eps^2 \Delta_x v + \pl_z^2 v &&\text{for } t>0,\ (x,z)\in
 \Omega=\Sigma {\ti} \III.
\end{align}
\end{subequations}
In this section we first prove existence of solutions for the initial-value
problem and then show that in the limit $\eps\to 0$ the corresponding solutions
$t\mapsto w_\eps(t) \in \bfH $ converge strongly to $t\mapsto
w_0(t)$ in the Hilbert space $\bfH$. For this it is sufficient to employ
the classical theory of Trotter and Kato, see e.g.\
\cite[Sec.\,3.3]{Pazy83SLOA}, where convergence of the resolvent
implies convergence of the semigroup.

\subsection{Formulation via strongly continuous semigroups}
\label{su:Semigroup}

By introducing the variable $V=\Dot U$ and setting $w=(U,V,v)$,  we
rewrite system \eqref{eq:SystemN} in the form $\Dot w = A_\eps w$ and
will show that the solutions $w$ can be obtained in the form $w(t)=
\ee^{t A_\eps} w_0$, i.e.\ we have to show that $A_\eps$ is the
generator of a strongly continuous semigroup on the space $\bfH=
\rmH^1(\Sigma)\ti \rmL^2(\Sigma)\ti \rmL^2(\Omega)$.  We define the unbounded linear
operators $A_\eps : D(A_\eps) \subset \bfH \to \bfH$ via
\begin{align*}
&D(A_\eps) = \bigset{ (U,V,v)\in \rmH^2(\Sigma) \ti
      \rmH^1(\Sigma) \ti \big(X_1^\eps(\Omega){\cup}
      Y^\eps(\Omega)\big) }{v|_{z=0}=V \text{ on } \Sigma}, \\
  &A_\eps \bma{c} U\\V\\v\ema =
  \bma{c} V \\ \Delta_x U - (\pl_z v)|_{z=0} \\
   \eps^2 \Delta_x v + \pl_z^2 v \ema.
\end{align*}
Here the spaces $X_\lambda^\eps(\Omega)$ with $\lambda>0$ and
$Y^\eps(\Omega)$ are defined via
\begin{align*}
  X_\lambda^\eps(\Omega)&:=
\bigset{v \in \rmL^2 (\Omega)}{  \eps^2\Delta_x v +\pl_z^2
    v - \lambda v=0, \ v|_{z\in 0}\in \rmH^1(\Sigma)} \ \text{ and}
  \\
  Y^\eps(\Omega)&:=\bigset{v\in \rmL^2(\Omega)}{ \eps^2 \Delta_x v +
    \pl_z^2 v \in \rmL^2(\Omega),\ v|_{z=0}=0} \,.
\end{align*}
We emphasize that the domain for $\eps=0$ is different from the
domains for $\eps>0$, because of the missing $x$-derivatives for $v$ in
the first case. Nevertheless, the trace of $v$ at $z=0$ is
well-defined in $\rmL^2(\Sigma)$ because $\pl_z v$ lies in
$\rmL^2(\Sigma, \rmH^1(\III))$ and $\rmH^1( \III )$
embeds continuously into $\rmC^0( \III )$.

More precisely, for $\eps>0$ we may apply the classical elliptic
regularity theory from \cite{LioMag72NHBV} which shows that
$Y^\eps(\Omega)$ is a closed subspace of $H^2(\Omega)$ whereas
$X_\lambda^\eps(\Omega)$ is only contained in $H^{3/2}(\Omega)$ but
not in $\rmH^2(\Omega)$. In Step 3 of the proof below, we  will show that
for $\eps>0$ we have
\begin{equation}
  \label{eq:XYlambda=XY1}
  X_\lambda^\eps(\Omega) \cup Y^\eps(\Omega) =
  X_1^\eps(\Omega) \cup Y^\eps(\Omega) \text{ for all }\lambda >0.
\end{equation}

For $\eps=0$ the spaces $X_\lambda^0(\Omega)$ and $Y^0(\Omega)$ have
lower regularity in $x \in \Sigma$, namely
\begin{align*}
  X_\lambda^0(\Omega)
  &= \bigset{v\in \rmH^1(\Omega)}{
                       v(x,z)=\ee^{\sqrt{\lambda} z}v(x,0),\
                       v(\cdot,0)\in \rmH^1(\Sigma) },
  \\
  Y^0(\Omega)
  &:=\bigset{ v \in \rmL^2\big(\Sigma;\rmH^2(\III)\big) }{
    v(x,0)=0 \text{ a.e.\ in } \Sigma}.
\end{align*}
Since $z\mapsto \ee^z - \ee^{\sqrt\lambda z}$ lies in
$H^2( \III ) $ and vanishes at $z=0$,  we easily see
$X_\lambda^0(\Omega) \cup Y^0(\Omega) = X_1^0(\Omega) \cup
Y^0(\Omega)$ for all $\lambda>0$.

Our first result in this section shows that for each
$\eps\geq 0$ the operator $A_\eps$ generates a strongly
continuous semigroup $(\ee^{tA_\eps})_{t\geq 0} $ on $\bfH $ with a
uniform growth rate 1.

\begin{theorem}[Generation of semigroups] \label{th:GenerSG}
For all $\eps \geq 0$ the operators $A_\eps$ defined above are
closed. For $\Re \lambda>0$ the resolvents $(A_\eps{-}\lambda I)^{-1}:
\bfH \to D(A_\eps)\subset \bfH$ exist and satisfy the estimate
\begin{equation}
  \label{eq:Resolvent}
  \big\| (A_\eps{-}\lambda I)^{-1} \big\|_{\bfH\to \bfH} \leq \frac 1
  {\lambda -1} \text{ for } \lambda >1.
\end{equation}
In particular, $A_\eps$ is the generator of the strongly continuous
semigroup $\ee^{t A_\eps}:\bfH\to \bfH$ satisfying $\| \ee^{t A_\eps}
\|_{\bfH\to \bfH} \leq 1+t\,\ee/2$ for $t\geq 0$. Moreover, the
functional
energy
\begin{equation}
  \label{eq:Energy}
  \calE_0(U,\Dot U,v)= \int_\Sigma \Big\{\frac12 \Dot U(x)^2 + \frac12
  |\nabla U(x)|^2 + \int_{-\infty}^0 \frac12\,v(x,z)^2 \dd z \Big\} \dd x
\end{equation}
is a Lyapunov function, i.e.\ along solutions we have the estimate
$\calE_0(U(t),\Dot U(t),v(t)) \leq \calE_0(U(s),\Dot U(s),v(s))$ for
$t > s \geq 0$.
\end{theorem}
\begin{proof} We first treat the case $\eps>0$ in Steps 1 to 3 and
  then discuss the differences for the case $\eps=0$ in Step 4.\smallskip

  \emph{Step 1: A priori estimate.} For $\alpha>0$ we use the norm
  $|\cdot|_\alpha$ on $\rmH^1(\Sigma)$ defined via $|U|_\alpha^2 =
  \alpha^2\|U\|_2^2 + \|\nabla_xU\|_2^2$, where $\|\cdot\|_2$ is the
  standard $\rmL^2$ norm on $\Sigma$.

For $w=(U,V,v)\in D(A_\eps)$ and $\alpha\geq 0$ we  obtain the
estimate
\begin{equation}
\label{eq:alpha.Estim}
\begin{aligned}
    \big\langle \!\!\big\langle A_\eps w,w
    \big\rangle\!\!\big\rangle_\alpha
  &:=\langle V,U\rangle_\alpha
      + \int_\Sigma\big(\Delta_xU{-}\pl_z v|_{z=0}\big) V \dd x
      + \int_\Omega \big(\eps\Delta_xv{+}\pl_z^2 v\big) v \dd z \dd x
  \\
  &= \int_\Sigma\Big\{ \alpha^2 UV+\nabla U\cdot \nabla V  - \nabla
    U\cdot \nabla V - (\pl_z v\, v)|_{z=0}\\
  &\hspace*{3em} - \int_{-\infty}^0\big(\eps^2 |\nabla_{x}v|^2 - (\pl_z
    v)^2\big) \dd z +   (\pl_z v\, v)|_{z=0} \Big\} \dd x
  \\
  &= \int_\Sigma  \alpha^2 UV \dd x - \int_\Omega \big( \eps^2
    |\nabla_{x}v|^2 - (\pl_z v)^2\big)  \dd z \dd x
  \\
  & \leq \frac\alpha2 \big( \alpha^2\|U\|_2^2 + \|V\|_2^2 \big)
    \leq \frac\alpha2 \big\langle\!\!\big\langle  w,w
    \big\rangle\!\!\big\rangle_\alpha =  \frac\alpha2 \norm w\norm_\alpha ^2,
  \end{aligned}
\end{equation}
where we used the norm $\norm w\norm_\alpha = \llla w,w,\rrra^{1/2}$.
For $\lambda >0$ and $F=(A_\eps {-}\lambda I)w$ we obtain the estimate
\[
  \norm F \norm _\alpha \norm w \norm _\alpha \geq  -\big\langle
  \!\!\big\langle F ,w
    \big\rangle\!\!\big\rangle_\alpha = -
   \big\langle \!\!\big\langle (A_\eps{-}\lambda I) w,w
                  \big\rangle\!\!\big\rangle_\alpha
 \geq \big( \lambda - \frac\alpha2 \big)
  \llla w,w \rrra_\alpha
  =\big( \lambda - \frac\alpha2 \big) \norm w\norm_\alpha ^2.
\]
Thus, for $\alpha>0$ we have established the estimate
\begin{equation}
  \label{eq:ResEst.alpha}
  \norm (A_\eps {-}\lambda I)^{-1} F\norm_\alpha \leq
  \frac1{\lambda{-}\alpha/2} \norm F\norm_\alpha \quad
   \text{for all } \lambda >\alpha/2.
\end{equation}
Because of $\norm \cdot \norm_1=\| \cdot\|_{\rmH^1}$ we obtain
\[
  \Big\| \big(A_\eps{-}\lambda I\big)^{-1} F \Big\|_\bfH = \big\| w
  \big\|_\bfH \leq \frac1{\Re \lambda -1/\sqrt2}  \big\| F \big\|_\bfH
  \,.\smallskip
\]
In particular, we have shown that the bounded linear operators
$A_\eps{-}\lambda I: D(A_\eps)
\to \bfH$ are injective. The following steps show that these operators
are also surjective, i.e.\ the resolvent equations have a solution in
$D(A_\eps)$.

\emph{Step 2: Reduction of resolvent equation.} It remains to show that for
all $\lambda >0$ the resolvent equation
$(A_\eps{-} \lambda I)w=F \in \bfH$ has a solution in $w=(U,V,v)\in
D(A_\eps)$. In this step we reduce the problem to an equation for $U$
alone.

Writing $F=(G,H,f)$ the system reads
\begin{subequations}
  \label{eq:ResolventEqn}
  \begin{align}
      \label{eq:Resolv.Sigma}
\text{on }\Sigma:\quad  & V-\lambda U= G ,\quad \Delta_x U-\pl_z v|_{z=0} -
                 \lambda V=H   , \quad V = v|_{z=0}, \\
  \label{eq:Resolv.Omega}
    \text{in } \Omega:\quad  & \eps^2 \Delta_x v + \pl_z^2 v - \lambda v = f .
\end{align}
\end{subequations}
Obviously, we can eliminate $V$ using the first equation giving
$V=G+\lambda U$.

Next, we solve \eqref{eq:Resolv.Omega} for $v$. Together with the
Dirichlet boundary condition $v=V$ at $z=0$, we obtain a unique
solution $v= \mathsf V_\lambda^\eps (f,V)$. By classical elliptic
regularity theory (see \cite{LioMag72NHBV}), for all $\lambda>0$ the
bounded linear operator $\mathsf V_\lambda^\eps$ maps
$\rmH^s(\Omega)\ti \rmH^{s+3/2}(\Sigma)$ to $\rmH^{s+2}(\Omega)$.
Because we have $V \in \rmH^1(\Sigma)$ and $f \in \rmL^2(\Omega)$, we
treat the two inhomogeneities separately, namely
$\mathsf V_\lambda^\eps (\cdot, 0): \rmL^2(\Omega)\to Y^\eps(\Omega)
\subset \rmH^2(\Omega)$ and
$\mathsf V_\lambda^\eps (0,\cdot): \rmH^1(\Sigma) \to X_\lambda^\eps(
\Omega) \subset \rmH^{3/2}(\Omega)$. As the equation for $U$ uses
$\pl_z v|_{z=0}$, we define the operators
\begin{equation}
    \label{eq:Dir2Neu.lambda}
     M_\lambda^\eps :\left\{\ba{ccc}
\rmL^2(\Omega) &\to& H^{1/2}(\Sigma) ,\\
f& \mapsto& \pl_z \mathsf V_\lambda^\eps (f,0)|_{z=0}, \ea\right.
\ \text{ and } \
 N_\lambda^\eps :\left\{\ba{ccc}
\rmH^1(\Sigma) &\to& \rmL^2(\Sigma) ,\\
V & \mapsto& \pl_z \mathsf V_\lambda^\eps (0,V)|_{z=0}, \ea\right.
\end{equation}
where $N^\eps_\lambda$ is a Dirichlet-to-Neumann operator.
It remains to solve an equation for $U$:
\begin{equation}
  \label{eq:Resolv.U}
  -\Delta_x U + (\lambda^2 I + \lambda N_\lambda^\eps )U =
   -H- M_\lambda^\eps f - (\lambda I + N_\lambda^\eps ) G.\smallskip
\end{equation}

\emph{Step 3. $(A_\eps {-}\lambda I)F \in D(A_\eps)$.}
Using $F=(G,H,f)\in \bfH$ and the mapping properties of $M_\lambda$
and $N_\lambda^\eps$, we see that the right-hand side in \eqref{eq:Resolv.U}
lies in $\rmL^2(\Sigma)$.  Moreover, for $\lambda>0 $, the operator on
the left-hand side
generates a bounded and coercive bilinear form on $\rmH^1(\Sigma)$,
because $\int_\Sigma U N_\lambda^\eps U \dd x = \int_\Omega \big(
\eps^2|\nabla_x v|^2 + (\pl_z v)^2 + \lambda v^2\big) \dd z \dd x \geq
0$, where $v= \mathsf V_\lambda(0,U)$. This is of course the same
calculation as in Step 1.   Thus, the Lax-Milgram theorem provides a
unique solution $U \in \rmH^1(\Sigma)$, which by classical linear
regularity lies even in $\rmH^2(\Sigma)$. From
$V=G+\lambda U$ we obtain $V \in \rmH^1(\Sigma)$. Finally, we obtain
$v= \mathsf V_\lambda^\eps (f,V) \in X^\eps_\lambda(\Omega) \cup
Y^\eps(\Omega)$.

Thus, we are done, if the identity \eqref{eq:XYlambda=XY1} is
established. For this we take any $W\in \rmH^1(\Sigma)$ and consider
$v_\lambda:=\mathsf V_\lambda^\eps (0,W) \in X_\lambda^\eps(\Omega)$
and   $v_1:=\mathsf V_1(0,W) \in X^\eps_1(\Omega)$. By the definition
of $\mathsf V_\lambda^\eps (0,\cdot)$ we see that the difference
$w:=v_\lambda -v_1$ satisfies the linear PDE
\[
  \eps^2 \Delta_x w+ \pl_z^2 w - w = (1{-}\lambda) v_\lambda \in
  \rmH^{3/2}(\Omega) , \quad w|_{z=0} = 0.
\]
Hence, we conclude $w \in Y^\eps(\Omega)$, which implies 
$v_\lambda = v_1 + w \in X^\eps_1(\Omega) \cup Y^\eps(\Omega)$ as
desired.\smallskip

\emph{Step 4. The case $\eps=0$.} The a priori estimate in Step 1
works for this case, too. The elimination of $V$ and $v$ works
similarly, but now with the simplification that $N_\lambda$ is
explicitly given, namely $N^0_\lambda = \sqrt \lambda I$. Together
with $X_\lambda^0(\Omega)\cup Y^0(\Omega) = X_1^0(\Omega)\cup
Y^0(\Omega) $ (see above) we conclude, and Theorem
\ref{th:GenerSG} is established.

\emph{Step 5. Growth rates for the semigroup.}  From
\eqref{eq:ResEst.alpha} we know that the semigroups $\ee^{t
  A_\eps}$ satisfy
the growth estimate $ \norm \ee^{t A_\eps} w \norm_\alpha \leq \ee^{t \alpha/2}
\norm w \norm_\alpha $ for all $\alpha\geq 0$. Setting
$\underline\alpha=\min\{1,\alpha\}$ and $\ol\alpha= \max\{1,\alpha\}$
and using the equivalence between $|\cdot|_\alpha$ and
$\|\cdot\|_{\rmH^1} = |\cdot|_1$, we obtain
\begin{align*}
  \|\ee^{t A_\eps} w\|_\bfH
  &\leq \frac1{\underline\alpha} \norm  \ee^{t A_\eps} w \norm_\alpha
\leq \frac1{\underline\alpha} \ee^{t\alpha/2} \norm w \norm_\alpha
    \leq \frac{\ol\alpha}{\underline\alpha}  \ee^{t\alpha/2}
    \|w\|_{\bfH} = \max\{ \alpha, 1/\alpha\}  \ee^{t\alpha/2}
    \|w\|_{\bfH}.
\end{align*}
Optimizing with respect to $\alpha>0$ yields the bound $\ee^{t/2}$ for
$t\in [0,2]$ and $t\,\ee /2$ for $t\geq 2$, which implies the final result $
\|\ee^{t A_\eps} \|_{\bfH \to \bfH} \leq (1+t\,\ee /2)$.

The final statement concerning $\calE_0$ follows by setting $\alpha
=0$, observing $\calE_0(w)=\frac12\norm w\norm_0^2$, and the contraction
property  $ \norm \ee^{t A_\eps} w \norm_0 \leq \ee^{0\cdot t}
\norm w \norm_0 = \norm w\norm_0$.
\end{proof}

\subsection{Convergence of semigroups}
\label{su:CvgSemigroup}

The next result proves the convergence of the resolvents
$(A_\eps{-}\lambda I)^{-1}$ as operators from $\bfH $ into itself in
the strong operator topology. The critical point is to understand the
convergence of the Dirichlet-to-Neumann operators $N_\lambda^\eps$ to
the limiting operator
$N_\lambda^0$, see \eqref{eq:Dir2Neu.lambda}.

\begin{proposition}[Strong convergence of resolvents]
  \label{pr:StrongResCvg}
  For all $\lambda>0$ and all $F \in \bfH$ we have the strong convergence
  $(A_\eps{-}\lambda I)^{-1} F \to (A_0{-}\lambda I)^{-1} F$.
\end{proposition}
\begin{proof} Throughout the proof $\lambda>0$ is fixed.

\emph{Step 1. Reduction to $F$ in a dense subset $\bfZ$ of $\bfH$.}
Let $\bfZ \subset \bfH$ be given such that $\bfZ$ is dense in $\bfH$
and that for all $F\in \bfZ$ we have $(A_\eps {-} \lambda I)^{-1}F \to
(A_0 {-} \lambda I)^{-1}F$ as $\eps \to 0^+$.

For an arbitrary $F\in \bfH$ we consider $F_n \in \bfZ$ with
$F_n \to F$ in $\bfH$ as $ n \to \infty$. By Step 1 in the proof of
Theorem \ref{th:GenerSG} we know that the resolvents
$(A_\eps {-} \lambda I)^{-1}$ are uniformly bounded by $C_\lambda$
with respect to $\eps >0$. Hence we have
\begin{align*}
  &\big\| (A_\eps {-} \lambda I)^{-1}F - (A_0 {-} \lambda I)^{-1}F
 \big\|_\bfH
  \\
  & \leq
  \big\| (A_\eps {-} \lambda I)^{-1}( F {-}F_n)  \big\|_\bfH
+\big\| (A_\eps {-} \lambda I)^{-1}F_n - (A_0 {-} \lambda I)^{-1}F_n
  \big\|_\bfH
    +\big\| (A_0 {-} \lambda I)^{-1} (F_n {-} F)  \big\|_\bfH
  \\
  &
  \leq C_\lambda \|F{-}F_n\|_\bfH +
  \big\| (A_\eps {-} \lambda I)^{-1}F_n - (A_0 {-} \lambda I)^{-1}F_n
    \big\|_\bfH   + C_\lambda \|F{-}F_n\|_\bfH .
\end{align*}
Thus, for a given $\delta>0$ we can make the difference small by first
choosing $n$ so big that $C_\lambda  \|F{-}F_n\|_\bfH < \delta/3$ and
then choosing $\eps_0 >0$ so small that the middle term is less than
$\delta /3$ for all $\eps \in {]0,\eps_0[}$ as well.
Thus, $(A_\eps {-} \lambda I)^{-1}F \to (A_0 {-} \lambda I)^{-1}F$
holds for all $F \in \bfH$.

\emph{Step 2. Higher regularity for smooth right-hand sides $F$.} We use
that the system is translation invariant in the domain $\Sigma$. Thus,
if the partial derivatives $\pl_{x_j} F$ lie in $\bfH$, then the
solutions $w^\eps = (A_\eps{-}\lambda I)^{-1} F$ have an additional
derivative in $x_j$ direction as well and satisfy the a priori
estimate $\| \pl_{x_j} w^\eps \|_\bfH \leq C_\lambda \|\pl_{x_j}
F\|_\bfH$. Thus, for $F$ in the dense subset
\begin{equation}
  \label{eq:def.bfZ}
    \bfZ = \bigset{ F\in \bfH}{ \nabla_x F \in \bfH }
\end{equation}
we obtain the improved estimate $\| w^\eps\|_\bfZ \leq C_\lambda
\|F\|_\bfZ $ where $\|F\|_\bfZ:=\|F\|_\bfH+\|\nabla_x F\|_\bfH$.

\emph{Step 3. Convergence for $\eps \to 0^+$.}  We now assume $F\in
\bfZ$ and compare $w^\eps =(U^\eps,V^\eps,v^\eps)= (A_\eps{-}\lambda
I)^{-1} F$ with $w^0=(U^0,V^0,v^0) = (A_0{-}\lambda I)^{-1}F$. As in
Step 1 of the proof of Theorem \ref{th:GenerSG}, we estimate the
difference $w^\eps - w^0$ as follows (choosing
$\alpha>0$ with $\sqrt\alpha <2\lambda$):
\begin{align*}
  & \big(\lambda{-}\frac{\sqrt\alpha}2\big)
    \big\langle\!\!\big\langle
        w^\eps{-}w^0 , w^\eps{-}w^0
    \big\rangle\!\!\big\rangle_\alpha
  \leq -  \big\langle\!\!\big\langle
    (A_0{-}\lambda I) (w^\eps{-}w^0), w^\eps{-}w^0
    \big\rangle\!\!\big\rangle_\alpha
  \\
  & \overset{*}= \big\langle\!\!\big\langle
    (0,0,\eps^2 \Delta_x v^\eps)^\top , w^\eps{-}w^0
         \big\rangle\!\!\big\rangle_\alpha
   = - \int_\Omega \eps^2 \nabla_x v^\eps( \nabla v^\eps{-}\nabla v^0)
    \dd z \dd x
  \\
  & \leq \wh C_\alpha \eps^2 \| w^\eps\|_\bfZ(\|w^\eps\|_\bfZ{+}\| w^0\|_\bfZ) \leq
    2\eps^2 \wh C_\alpha C_\lambda \|F\|_\bfZ^2
    \ \to 0 \ \text{ \ as } \eps \to 0^+\,.
\end{align*}
In the identity $\overset*=$ we have used the cancellation arising
from $(A_0{-}\lambda I)w^0=F$ and
\[
  (A_0{-}\lambda I) w^\eps =  (A_\eps{-}\lambda I) w^\eps  +
  (A_0{-}A_\eps) w^\eps = F -(0,0,\eps^2 \Delta_x v^\eps)^\top \, .
\]
By the equivalence of the $\bfH$ norm and the norm induced by
$\big\langle\!\!\big\langle   \cdot , \cdot
\big\rangle\!\!\big\rangle_\alpha$, we conclude $\|
w^\eps{-}w^0\|_\bfH$ $ = C \,\eps$, and Proposition
\ref{pr:StrongResCvg} is proved.
\end{proof}

Theorem \ref{th:GenerSG} and Proposition \ref{pr:StrongResCvg} are the
basis for the following result that states
that the contraction semigroups $(\ee^{t A_\eps})_{t\geq 0}$ on $\bfH$
converge as $\eps \to 0^+$ in the strong operator topology. Indeed, the proof
of the first part is a direct consequence of the Trotter--Kato theory, see
\cite[Sec.\,3.3]{Pazy83SLOA}, while the second part uses explicit
estimates.

\begin{theorem}[Strong convergence of the solutions]
  \label{th:StrongCvgSol} Consider the operators $A_\eps$ defined in
  Theorem \ref{th:GenerSG} and the induced contraction
  semigroups $ \ee^{t A_\eps}: \bfH \to \bfH$ for
  $t\geq 0$. Then, for all initial conditions $w_0 \in \bfH$, the
  solutions
  $w^\eps: {[0,\infty[} \to \bfH,\ t\mapsto w^\eps(t)=\ee^{t A_\eps} w_0$ satisfy
  for all $t\geq 0$ the convergence $w^\eps(t) \to w^0(t)$ as
  $\eps \to 0$.

  Moreover, for initial conditions with additional derivatives in
  $x$-direction, namely $w_0 \in \bfZ$ (cf.\ \eqref{eq:def.bfZ}) we
  have the quantitative error estimate
\begin{equation}
  \label{eq:QuantEstim}
  \| w^\eps(t) - w^0(t)\|_\bfH \leq \eps\,\sqrt{t}\: (2.3{+}t)^2
  \,\|w_0\|_\bfZ  \quad
  \text{for all }t\geq 0.
\end{equation}
\end{theorem}
\begin{proof} It remains to show \eqref{eq:QuantEstim}. For this we
  set $\delta = w^\eps-w^0$ and perform a simple energy estimate,
  where we use that $w^\eps=(U^\eps,V^\eps,v^\eps)$ and
  $w^0=(U^0,V^0,v^0)$ are sufficiently smooth solutions of
  \eqref{eq:SystemN}, because we have the extra regularity of 
  $w_0\in \bfZ$. We employ the norms $\norm\cdot\norm_\alpha$ as
  defined in \eqref{eq:alpha.Estim} and find
  \begin{align*}
    \frac12 \frac{\rmd}{\rmd t} \norm \delta \norm_\alpha^2
    &= \int_\Sigma \!\alpha^2(U^\eps{-}U^0)(V^\eps{-}V^0)\dd x
      - \! \int_\Omega \!\! \big\{|\pl_z v^\eps {-} \pl_z v_0|^2 + \eps^2
      \nabla_xv^\eps\cdot (\nabla_x v^\eps{-}\nabla_x v^0)\big\}\dd
      z \dd x
    \\
    &\leq \frac\alpha2 \norm \delta\norm^2_\alpha - 0
      + \eps^2 \|\nabla_x v^\eps\|_{\rmL^2(\Omega)}
      \big( \|\nabla_x v^\eps\|_{\rmL^2(\Omega)}+
      \|\nabla_x v^0\|_{\rmL^2(\Omega)}\big)
    \\
    &\leq \frac\alpha2 \norm \delta\norm^2_\alpha + \eps^2 \|w_0\|_\bfZ^2
      \:2\,\big(1 + t\,\ee/2\big)^2,
  \end{align*}
where we used $\|\nabla_x w^\eps\|_\bfH  \leq \|\ee^{t A_\eps}\nabla_x
w_0\|_\bfH\leq \|\ee^{t A_\eps}\|_{\bfH\to\bfH} \|w_0\|_\bfZ $ and the growth
estimate from Theorem \ref{th:GenerSG}.  Using $\delta(0)=w_0-w_0=0$,
the Gronwall lemma yields
\[
   \norm \delta(t) \norm_\alpha^2 \leq \eps^2 \int_0^t 2\,
   \ee^{\alpha(t{-}s)} (1{+}s\,\ee /2)^2 \dd s \:\|w_0\|_\bfZ^2.
\]
For $t\in [0,2]$ we choose $\alpha=1$ and obtain
\[
  \|\delta(t)\|_\bfH^2 = \norm \delta(t) \norm_1^2 \leq C_* t\,\eps^2
  \|w_0\|_\bfZ^2 \text{ for  }t\in [0,2]
\]
where $C_*=\int_0^2 \ee^{2-s}(1{+}s\ee /2)^2 \dd s\approx 27.14...\leq
2.3^4$. For $t\geq 2$ let $\alpha=2/t\leq 1$ to obtain
\begin{align*}
\|\delta(t)\|_\bfH^2 &\leq \frac1{\alpha^2} \norm
\delta(t)\norm_\alpha^2 \leq \frac{t^2}4 \int_0^t 2\,\ee^{2(1-s/t)}
                       (1{+}s\ee /2)^2 \dd s \: \eps \|w_0\|_\bfZ^2
  \\
  & =\frac{t^3}{32}\Big((\ee^2{-}5)\ee^2t^2+4(\ee^2{-}3)\ee
    t+8(\ee^2{-}1) \Big)\: \eps^2 \|w_0\|_\bfZ^2 \ \leq \frac t6\big( 1+
    t\,\ee/2)^4 \: \eps^2 \|w_0\|_\bfZ^2.
\end{align*}
Combining this with the result for $t\in [0,2]$ and using
$(\ee/2)^4\leq 6$,  we arrive at
$\|\delta(t)\|_\bfH^2 \leq t\, (2.3{+}t)^4 \,\eps^2   \|w_0\|_\bfZ^2$ for
all $t\geq 0$, which is the desired result \eqref{eq:QuantEstim}.
\end{proof}

\section{Energy and dissipation functionals}
\label{se:EnergyDiss}

We now show that the fractionally-damped wave equation \eqref{eq:FDWE}
carries a natural energy-dissipation structure. This is done in two
different ways. First, we reduce the natural energy-dissipation
structure of the limiting system \eqref{eq:SystemN} with $\eps=0$ by
eliminating the diffusion equation. For this we first study the
one-dimensional diffusion equation $\Dot v= \pl_z^2 v$ on the half
line $\III$ in detail. Second, we show by a direct calculation that
the energy-dissipation structure extends to a more general class of
fractionally-damped wave equations, where the time derivative of order
3/2 is replaced by order $1{+}\alpha $ with $\alpha \in {]0,1[}$.

\subsection{Diffusion equation on the half line}
\label{su:HalfLine}

We consider the following initial-boundary value problem on
$I:={]{-}\infty, 0[}$:
\begin{equation}
  \label{eq:HalfLine}
  \Dot v= \pl_z^2v, \quad v(0,z)=v_0(z), \quad v(t,0)= \varphi(t).
\end{equation}
We always assume the compatibility condition $v_0(0)=\varphi(0)$.


Using the one-dimensional heat kernel $H(t,y) =
(4\pi t)^{-1/2} \exp\!\big({-}y^2/(4t)\big)$ and the reflection principle,
the influence of $v_0$ is described via
$K_\mafo{Dir}(t,z,y)= H(t,z{-}y)-H(t,z{+}y)$ such that homogeneous Dirichlet
data follow from $ K_\mafo{Dir}(t,0,y)=0$:
\[
v(t,z)= \int_I K_\mafo{Dir} (t,z,y) v_0(y) \dd y.
\]

To obtain the influence of the
inhomogeneous Dirichlet data at $z=0$, we set $v_0=0$ and make the
ansatz $v(t,z)=  w(t,z)+\varphi(t)$ such that $w$ has to satisfy
\[
w_t = w_{zz} -\Dot \varphi(t) , \qquad w(t,0)=0, \quad w(0,z)=-\varphi(0)=0.
\]
With Duhamel's principle (variation-of-constants formula) we obtain
\[
w(t,z)= -\int_0^t \int_I K_\rmD(t{-}s,z,y)
\Dot \varphi(s)\dd y \dd s.
\]
Setting $G(y):=\int_{-\infty}^y H(1,\eta) \dd \eta$ (such that
$G(-\infty)=0$ and $G(\infty)=1$) we obtain
\[
  w(t,z)=   \int_0^t \Dot \varphi(s) \big(G(z/\sqrt{t{-}s})
  - G({-}z/\sqrt{t{-} s})\big) \dd s  .
\]
Putting both cases together, the full solution formula for
\eqref{eq:HalfLine} reads
\[
v(t,z) = \int_I K_\mafo{Dir}(t,z,y) v_0(y) \dd y +\varphi(t)
+\int_0^t  \Dot \varphi(s) \big(G(z/\sqrt{t{-}s}) - G({-}z/\sqrt{t{-}
  s})\big) \dd s  .
\]
%
%
For the analysis related to the fractionally-damped wave equation, we
consider only the case $v_0 \equiv 0$, which implies
$\varphi(0)=0$ as well by continuity of the boundary-initial data.
Doing integration by parts for the time integral and using
\[
  \pl_s \big({\mp} G(\pm z/\sqrt{t{-}s})\big)= H(1,\pm z/\sqrt{t{-}s})
  \frac z{2(t{-}s)^{3/2}} = -\pl_z H(t{-}\tau, z)
\]
we arrive, for the case $v_0\equiv 0$, at the relation
\begin{equation}
  \label{eq:DiffHL3}
  v(t,z) = \int_0^t K_0(t{-}\tau,z)\, \varphi(\tau) \dd \tau \quad
  \text{with }  K_0(t,z)= 2\pl_z H(t,z).
\end{equation}
Using $\pl_z^2 H = \pl_t H$ and doing another integration by parts
(using $\varphi(0)=0$ again) we find
\begin{equation}
  \label{eq:DiffHL4}
  \pl_z v(t,z) = \int_0^t K_1(t{-}\tau,z) \,\Dot \varphi(\tau) \dd \tau
   \quad \text{with } K_1(t,z)= 2 H(t,z).
\end{equation}
In particular, evaluating at $z=0$, where $H(t,0)=1/\sqrt{4 \pi t}$, we find
\begin{equation}
  \label{eq:DiffHL5}
  \pl_z v(t,0)= \int_0^t \frac{\Dot\varphi(\tau)}
    {\sqrt{\pi (t{-}\tau)}} \:\dd \tau .
\end{equation}
According to the definition \eqref{eq:Caputo}, the boundary derivative
$\pl_z v$ is the
fractional Caputo derivative of order $1/2$ of $\varphi$, i.e.\
$\pl_z v(\cdot,0) = {}^\rmC\rmD^{1/2} \varphi$.\smallskip

We now derive an energy-dissipation balance for the diffusion equation
by rewriting the natural $\rmL^2$ integrals in terms of the boundary
value $\varphi$. The starting point is
the classical relation
\begin{equation}
  \label{eq:EnerDiss.v}
    \frac\rmd{\rmd t} \int_I\, \frac12\,v(t,z)^2 \dd z = \int_I v\Dot v\dd z =
  \int_I v \pl_z^2 v \dd z = v(t,0)\pl_z v(t,0) - \int_I \big(\pl_z
  v(t,z)\big)^2 \dd z.
\end{equation}

For solutions $v$ of \eqref{eq:HalfLine} with $v_0\equiv 0$ and
$\varphi(0)=0$ we can rewrite this energy-dissipation balance totally
in terms of $\varphi$ by using the following result.

\begin{proposition}
  \label{pr:EnerDiss.varphi}
Assume that $v$ is given via \eqref{eq:DiffHL3} and $\pl_z v$ by
\eqref{eq:DiffHL4}, then  we can express twice
the energy $\int_I v^2 \dd z$
and the dissipation $\int_I (\pl_z v)^2 \dd z$ via
\begin{subequations}
  \label{eq:M.j}
  \begin{align}
 \label{eq:M.j.a}
 \int_I  v^2 \dd z \quad &= \int_0^t \int_0^t
 M_0(t{-}r,t{-}s)\,\varphi(r)\varphi(s) \dd r \dd s \quad \text{and}
\\
    \label{eq:M.j.b}
    \int_I (\pl_z v)^2 \dd z &= \int_0^t \int_0^t
 M_1(t{-}r,t{-}s)\,\Dot \varphi(r) \Dot\varphi(s) \dd r \dd s \quad
    \\
     \label{eq:M.j.c}
\text{ where } &M_j(r,s)= \int_I K_j(r,z)K_j(s,z) \dd z =
                 \frac{2^j}{\sqrt{4\pi}\,(r{+}s)^{3/2-j}} .
\end{align}
In particular, $M_j(r,s)=\wt M_j(r{+}s)$ and
$\pl_r M_1(r,s)=\pl_sM_1(r,s)=- M_0(r,s)$.
\end{subequations}
\end{proposition}
\begin{proof} The relations \eqref{eq:M.j.a} and \eqref{eq:M.j.b} with
  $M_j(r,s)= \int_I K_j(r,z)K_j(s,z) \dd z $ follow simply by the
  definitions. Thus, it remains to establish the explicit formulas for
  $M_0$ and $M_1$ by exploiting the structure of $K_0= 2 \pl_z H =
  -\frac{z}{2t}\, 2H $ and
  $K_1 = 2H$. We obtain
\[
K_j(r,z)K_j(s,z) = 4 \Big( \frac{z^2}{4rs}\Big)^{1-j}\frac1{4\pi
  \sqrt{rs}} \exp\Big( {-}\frac{z^2}{4r} - \frac{z^2}{4s}\Big) = \Big(
\frac{z^2}{4rs}\Big)^{1-j} \frac1{\pi\sqrt{rs}} \exp\Big(
{-}\frac{r{+}s}{4rs}\,z^2\Big)  .
\]
An explicit integration  with $\int_0^\infty \ee^{-a^2
  x^2} \dd x =\sqrt\pi /(2a)$ and  $\int_0^\infty x^2\ee^{-a^2
  x^2} \dd x =\sqrt\pi /(4a^3)$ yields the stated formulas for $M_0$ and
$M_1$.
\end{proof}

It is a surprising fact that $M_0$ and $M_1$ depend only on the sum
$r{+}s$ rather than on $r$ and $s$ individually.  The relations in
\eqref{eq:M.j} allow us to rewrite the energy-dissipation balance
\eqref{eq:EnerDiss.v} in terms of $\varphi$ alone. We obtain the
identity
\begin{align}
  \nonumber
&\frac\rmd{\rmd t} \int_0^t\int_0^t \frac12 M_0(t{-}r,t{-}s)\,
                 \varphi(s) \varphi(r) \dd r \dd s\\
 \label{eq:EnDiFract1/2} &= \varphi(t) \int_0^t
\frac{\Dot \varphi(\tau)}{\sqrt{\pi(t{-}\tau)}} \dd \tau
-  \int_0^t\int_0^t M_1(t{-}r,t{-}s)\,
\Dot \varphi(s) \Dot \varphi(r) \dd r \dd s
\end{align}

\subsection{An energy-dissipation relation for fractional derivatives}
\label{su:EnerDissFract}

Here we show that the identity \eqref{eq:EnDiFract1/2} can be derived
in an independent way, not using the Dirichlet-to-Neumann map for
the one-dimensional diffusion equation. We even generalize the result
to the case of general fractional derivatives
${}^\rmC\rmD^\alpha\varphi$, where \eqref{eq:EnDiFract1/2} is the
special case $\alpha=1/2$.  For this we set
\begin{equation}
  \label{eq:def.N.j}
  N_0^\alpha(r,s)= \frac{\alpha}{\Gamma(1{-}\alpha)
    \,(r{+}s)^{1+\alpha}} \quad \text{and}  \quad
  N_1^\alpha(r,s)= \frac{1}{\Gamma(1{-}\alpha)\,(r{+}s)^{\alpha}}.
\end{equation}
With this, we obtain the following result.

\begin{proposition}
  \label{pr:identity}
For all $\varphi\in \rmC^1({[0,\infty[})$ with $\varphi(0)=0$ we have
the identity
\begin{align*}
&\frac\rmd{\rmd t} \int_0^t\!\!\int_0^t \frac12 N^\alpha_0(t{-}r,t{-}s)\,
                 \varphi(s) \varphi(r) \dd r \dd s
 \\  &
  = \varphi(t)\: {}^\rmC\rmD^\alpha \varphi(t)
-  \int_0^t\!\!\int_0^t N^\alpha_1(t{-}r,t{-}s)\,
\Dot \varphi(s) \Dot \varphi(r) \dd r \dd s.
\end{align*}
\end{proposition}
\begin{proof}
  We set $r=t{-}\rho$ and $s=t{-}\sigma$ and obtain
\begin{align*}
&\frac\rmd{\rmd t} \int_0^t\int_0^t \frac12 N^\alpha_0(t{-}r,t{-}s)\,
                 \varphi(s) \varphi(r) \dd s \dd r
 = \frac\rmd{\rmd t} \int_0^t\int_0^t \frac12 N^\alpha_0(\rho,\sigma )\,
                 \varphi(t{-}\sigma) \varphi(t{-}\rho) \dd \sigma \dd \rho
  \\
  &\overset1=\int_{\rho=0}^t \int_{\sigma=0}^t \frac12 N^\alpha_0(\rho,\sigma)
    \big( \varphi(t{-}\sigma) \Dot \varphi(t{-}\rho) +
     \big( \Dot \varphi(t{-}\sigma) \varphi(t{-}\rho) \big)
    \dd \sigma \dd \rho.
  \\
  &\overset2=
    \int_{\rho=0}^t \int_{\sigma=0}^t N^\alpha_0(\rho,\sigma)
    \varphi(t{-}\sigma) \Dot\varphi(t{-}\rho) \dd \sigma \dd \rho.
\end{align*}
Here $\overset1=$ uses $\varphi(0)=0$ such that the boundary terms
arising from $\frac\rmd{\rmd t} \int_0^t g(\tau)\dd \tau = g(t)$
vanish. In $\overset 2= $ we simply use the symmetry
$N^\alpha_0(r,s)=N^\alpha_0(s,r)$.

In the next step we perform an integration by parts with respect to
$\sigma\in [0,t]$ and use the fundamental relation
$N^\alpha_0(\rho,\sigma) = - \pl_\sigma
N^\alpha_1(\rho,\sigma)$. Hence, we continue
\begin{align*}
&=
 \int_{\rho=0}^t \Big\{
      \big[{-}N^\alpha_1(\rho,\sigma)\varphi(t{-}\sigma)\big]\Big|_0^t
      -\int_0^t  \big({-}N^\alpha_1(\rho,\sigma)\big)
      \big({-}\Dot\varphi(t{-}\sigma)\big)
                 \dd \sigma \Big\}
                 \Dot \varphi(t{-}\rho) \dd \sigma \dd \rho
  \\
  &=\int_{\rho=0}^t
  N^\alpha_1(\rho,0)\varphi(t)\Dot \varphi(t{-}\rho) \dd \sigma \dd \rho
  -\int_0^t \int_0^t N^\alpha_1(\rho,\sigma)\Dot\varphi(t{-}\sigma)
  \Dot \varphi(t{-}\rho) \dd \sigma \dd \rho ,
\end{align*}
where we again used $\varphi(0){=}0$. The definition of $N^\alpha_1$
gives $N^\alpha_1(t{-}\tau,0)= 1/ \big(\Gamma(1{-}\alpha) (t{-}\tau)^\alpha
\big) $, such that the first term is indeed equal to
$\varphi\;{}^\rmC\rmD^\alpha \varphi$. With this, the result is
established.
\end{proof}

We emphasize that the above result does not need the exact form of
$N_0^\alpha$ and $N_1^\alpha$ as given in \eqref{eq:def.N.j}. We only
exploited the relations
\[
  N^\alpha_0(r,s)= N^\alpha_0(s,r), \quad
  N^\alpha_0(r,s)= - \pl_s  N^\alpha_1(r,s), \quad
  N^\alpha_1(r,0)=1/\big( \Gamma(1{-}\alpha)\,r^\alpha\big).
\]
Clearly, there are many more functions satisfying these
conditions. However, we also want positive semi-definiteness of the kernels
$N^\alpha_j$, i.e.,
\[
\int_0^t \int_0^t N^\alpha_j(r,s) \psi(s)\psi(r) \dd s \dd r \geq 0
\quad \text{for all }\psi\in \rmC^0({[0,\infty[}),\ t>0,
\text{ and }j\in \{0,1\}.
\]

For general $N^\alpha_j$ this positive semi-definiteness is a
significant restriction, but for our chosen cases it can be
established as follows:
\begin{align*}
  0&\leq \int_{y=0}^\infty \Big(\int_{r=0}^t
      \frac{y^{\alpha-1/2}} {r^\alpha}
     \ee^{-y^2/r} \psi(r)\dd r \Big)^2 \dd y
  \\
   &= \int_{y=0}^\infty \int_{r=0}^t\int_{s=0}^t
       \frac{y^{\alpha-1/2}} {r^\alpha}
     \ee^{-y^2/r} \psi(r)  \frac{y^{\alpha-1/2}} {s^\alpha}
    \ee^{-y^2/s} \psi(s) \dd s \dd r  \dd y
  \\
  &=  \int_{r=0}^t\int_{s=0}^t  \int_{y=0}^\infty \frac{y^{2\alpha-1}}{(rs)^\alpha}
     \ee^{- y^2(r{+}s)/(rs)} \dd y \; \psi(r)\psi(s) \dd s \dd r.
\end{align*}
Using $\int_0^\infty |y|^{2\alpha-1} \ee^{-by^2} \dd y =
\Gamma(\alpha)/(2 b^\alpha)$, we obtain the desired result
\[
  0 \leq \int_{r=0}^t\int_{s=0}^t \frac{\Gamma(\alpha)}
  {(r{+}s)^\alpha} \:\psi(r)\psi(s) \dd s \dd r
  \quad \text{for all }\psi\in \rmC([0,t]),
\]
which holds for all $\alpha \geq 0$.

\subsection{Energetics for the fractionally-damped wave equation}
\label{su:EnerFDWE}

To derive the physically relevant energy-dissipation balance for the
the fractionally-damped wave equation
\begin{equation}
  \label{eqB:FDWE1}
  \DDot U(t,x) + \int_0^t \frac{\DDot U(s,x)}{\sqrt{\pi(t{-}s)}} \dd s
   \ = \ \Delta U(t,x)
\end{equation}
we use the limiting system \eqref{eq:LiSystem}. The latter is a classical
system of partial differential equations, and it is easy to write down
the physically motivated energy functional $\calE$ and the
corresponding dissipation function $\calD$.

The total energy $\calE$ is the sum of the kinetic and potential
energy in the membrane $\Sigma$ plus the kinetic energy in the lower
half space $\Omega = \Sigma \ti \III$, where we consider $v$ as the
horizontal component of a shear flow. This leads to $\calE_0$ as
defined in \eqref{eq:Energy}, and along the solutions of
\eqref{eq:LiSystem} the energy-dissipation balance takes the form
\[
\frac{\rmd}{\rmd t}  \calE_0(U(t),\Dot U(t), v(t)) = -\calD(U,\Dot
U,v):= \int_\Omega \frac12 (\pl_zv)^2 \dd z \dd x.
\]
This shows that the only dissipation occurs by the (shear) viscosity of the
fluid in the lower half space $\Omega$.

As explained in Section \ref{su:LimitMOdel} and \ref{su:HalfLine} we
can eliminate $v$ via (using $\varphi(t,x)=\Dot U(t,x,0)$ and assuming
$v(0,x,z)=0$ of all $(x,z)\in \Omega$)
\[
  v(t,x,z)=\int_0^t 2 \pl_z H(t{-}\tau,z) \Dot U(\tau,x) \dd \tau,
  \quad  \text{where } H(t,z)=\frac{ \ee^{-z^2/(4t)}} {\sqrt{4\pi t}}\:.
\]
This allows us to eliminate $\pl_z v(t,x,0)$ via \eqref{eq:DiffHL5} and we obtain
the fractionally-damped wave equation \eqref{eqB:FDWE1}.

Using the formulas derived in Proposition \ref{pr:EnerDiss.varphi} we
obtain the reduced energy function $\bfE$ and the reduced dissipation
function $\bfD$ in the form
\begin{subequations}
\begin{align}
  \nonumber
  \bfE(U(t),[\Dot U]_{[0,t]})&:= \int_\Sigma \Big\{\frac12\Dot U(t)^2
                               + \frac12|\nabla_xU(t)|^2 \\
 &\qquad \quad \ +
   \int_0^t\!\!\int_0^t \frac1{4\sqrt\pi\,(2t{-}r{-}s)^{3/2}} \:
   \Dot U(r,x) \Dot U(s,x) \dd r\dd s \Big\} \dd x ,\\
  \bfD(U(t),[\Dot U]_{[0,t]})&:=  \int_\Sigma \Big\{
   \int_0^t\!\!\int_0^t \frac1{\sqrt\pi\,(2t{-}r{-}s)^{1/2}}\:
    \DDot U(r,x) \DDot  U(s,x) \dd r\dd s \Big\}\dd x    .
\end{align}
\end{subequations}
Clearly, along solutions of the fractionally-damped wave equation
\eqref{eqB:FDWE1} we have the reduced energy-dissipation balance
\begin{equation}
  \label{eq:ReducEnerDiss}
    \frac{\rmd}{\rmd t} \bfE(U(t),[\Dot U]_{[0,t]})
  = - \bfD(U(t),[\Dot U]_{[0,t]}).
\end{equation}
Of course, it is possible to check this identity directly without any
reference to the limiting system \eqref{eq:LiSystem} involving the
hidden state variable $v$. For this we do the standard
argument for energy conservation for the wave equation plus the
calculation in the proof of Proposition \ref{pr:identity} for the
parts non-local in time.

In the related works \cite{VerZac08LFCS, GYCCEO15MMSD, VerZac15ODET}
other energy functionals where constructed for equations with
fractional time derivatives. However, the approach there is quite
different and is less inspired by the true energy and dissipation
hidden in the eliminated state variable $v$.
\smallskip

Indeed, we may generalize the energy-dissipation balance
\eqref{eq:ReducEnerDiss} to the case of fractional damping of order
$1{+}\alpha \in {]1,2[}$. We consider \eqref{eqB:FDWE1} as a special case of the
equation
\begin{equation}
  \label{eqB:FDWE.alpha}
  \DDot U + {}^\rmC\rmD^\alpha \Dot U = \Delta U \quad \text{ on } \Sigma.
\end{equation}
Taking into account the calculations in Section \ref{su:EnerDissFract}
we define the energy $\bfE^\alpha$ and the dissipation function
$\bfD^\alpha$ via
\begin{subequations}
  \label{eqbfE.bfD.alpha}
\begin{align}
    \bfE^\alpha(U(t),[\Dot U]_{[0,t]})&:= \int_\Sigma \Big\{\frac12\Dot U(t)^2
                               + \frac12|\nabla_xU(t)|^2 \\
\nonumber
 &\qquad \quad \ +
   \int_0^t\!\!\int_0^t
   \frac{\alpha/2}{\Gamma(1{-}\alpha)(2t{-}r{-}s)^{1{+}\alpha}}  \:
   \Dot U(r,x) \Dot U(s,x) \dd r\dd s \Big\} \dd x ,
\\
  \bfD^\alpha(U(t),[\Dot U]_{[0,t]})&:=  \int_\Sigma \Big\{
   \int_0^t\!\!\int_0^t \frac1{\Gamma(1{-}\alpha)(2t{-}r{-}s)^{\alpha}}\:
    \DDot U(r,x) \DDot  U(s,x) \dd r\dd s \Big\}\dd x    .
\end{align}
\end{subequations}
Clearly, for sufficiently smooth solutions of the
fractionally-damped wave equation
\eqref{eqB:FDWE.alpha} we have the reduced energy-dissipation balance
\begin{equation}
  \label{eq:ReducEnerDiss.alpha}
    \frac{\rmd}{\rmd t} \bfE^\alpha(U(t),[\Dot U]_{[0,t]})
  = - \bfD^\alpha(U(t),[\Dot U]_{[0,t]}).
\end{equation}

\section{Conclusion and outlook}
\label{se:Concl.Outl}

In this work we have shown that the fractionally-damped wave equation
can be obtained as a scaling limit from a bulk-interface coupling
between a wave equation for a membrane and a viscous fluid motion in
the adjacent half space. The coupling is such that the natural
mechanical energies act as a Lyapunov function. We have identified the
physical scaling parameters like the equivalent membrane thickness
$\ell_\mafo{thick}=\rhoo/\rhooo$ for the vertical scaling and the
effective travel length  $\ell_\mafo{trav}(t_*)
=\rhoo^{3/2}\kkk^{1/2}/(\mu \rhooo)$ for the horizontal scaling. Thus, taking
the limit $\eps \to 0$ in the critical parameter
\[
  \eps = \frac{\ell_\mafo{thick}}{\ell_\mafo{trav}(t_*)} =
  \frac{\mu}{\sqrt{\rhoo\, \kkk\,}}
\]
leads to the appearance of the fractionally-damped wave equation.

The first main outcome of
the mathematical analysis is that the system is stable uniformly with
respect to $\eps$ and that it converges strongly in the natural
energy space $\bfH$ in the sense of linear semigroup theory. For
initial data with higher horizontal regularity a convergence rate
could be derived.
Thus, the fractional time derivative of order $3/2$ appears naturally
as a consequence of the  Dirichlet-to-Neumann map of a one-dimensional
parabolic equation on the half line.

The second outcome of our approach is the energy-dissipation structure
for the fractio\-nally-damped wave equation which is derived by
integrating out the ``hidden states'' $v$ in the fluid layer in the
full mechanical energy-dissipation structure of the coupled system of
partial differential equations. As expected, we obtain quadratic
functionals for the reduced energy and the reduced dissipation
function that are non-local in time, thus keeping track of information
stored in the hidden state variable $v$. It is surprising that both
quadratic functionals obtained have memory kernels $M_j(t{-}r,t{-}s)$
that depend only on the sum $(t{-}r)+(t{-}s)$. It is certainly
important to understand where this special structure comes from and
how it relates to more general energy-dissipation structures as
introduced in \cite{VerZac08LFCS, GYCCEO15MMSD, VerZac15ODET}.

A major restriction occurs through our assumption $v(0,x,z)=0$ for
a.a.\ $(x,z) \in \Omega$, which implies $\Dot U(0,x)=0$. We expect
that this assumption can be avoided by suitably generalizing the
Caputo derivative and by extending the memory kernel to negative time,
thus allowing for some pre-initial conditions. This will be the
content of further research.

This work is understood as a first step to understand the principles
behind damping based on fractional time-derivatives. In subsequent
works we plan to extend the analysis to a more physical model, namely
that of a true membrane over a viscous incompressible fluid governed
by the Navier-Stokes equations. The approach based on partial
differential equations developed here, will then allow us to study the
full vector-valued case $v(t,x,z)\in \R^d$ including the associated
nonlinearities.  It will be interesting to see under what conditions
the relevant scalings in the nonlinear setting will be the same as in
the linear theory in \cite{KSSN17NFWE, KSSN17NFWEs}. Moreover, it will
be critical to see the occurrence of fractional damping, which relies
on the linearity of the Dirichlet-to-Neumann map of the parabolic
equation on the vertical half line.

\footnotesize


   
\newcommand{\etalchar}[1]{$^{#1}$}
\def\cprime{$'$}

\end{document}